\documentclass{article}
\usepackage{amsfonts}
\usepackage{amsmath}
\usepackage{natbib}
\usepackage{authblk}

\newtheorem{theorem}{Theorem}

\newtheorem{proposition}[theorem]{Proposition}

\newenvironment{proof}[1][Proof]{\noindent\textbf{#1.} }{\ \rule{0.5em}{0.5em}}

\begin{document}

\title{A note on the relation between one--step, outcome regression and IPW--type estimators of
parameters with the mixed bias property}
\author[1]{Andrea Rotnitzky}
\author[2]{Ezequiel Smucler}
\author[3]{James M. Robins}

\affil[1]{University of Washington, Department of Biostatistics}
\affil[2]{Glovo}
\affil[3]{Harvard T.H. Chan School of Public Health}

\maketitle

\begin{abstract}
\cite{bruns2025augmented}\ established an algebraic identity between the
one-step estimator and a specific outcome regression--type estimator for a
class of parameters that forms a strict subset of the class introduced in 
\cite{chernozhukovclass}, assuming both nuisance functions are estimated as
linear combinations of given features. They conjectured that this identity
extends to the broader mixed bias class introduced in \cite{characterization}%
. In this note, we prove their conjecture and further extend the result to
allow one of the nuisance estimators to be non-linear. We also relate these
findings to the work of \cite{robins2007comment}, who established other identities
linking one-step estimators to outcome regression--type and IPW--type estimators.
\end{abstract}

\section{Introduction}

\cite{bruns2025augmented} presented an algebraic identity between one-step
estimators and a specific outcome regression-type estimator for a strict
subset of the parameter class introduced in \cite{chernozhukovclass}. Their results assumed that the one-step estimator uses estimators of the nuisance parameters that are arbitrary linear combinations of features. They conjectured that
their finding could be extended to the broader class of parameters with the
mixed bias property, introduced in \cite{characterization}. In this note, we
prove this conjecture. In doing so, we additionally extend the results in 
\cite{bruns2025augmented} to allow for possibly non-linear estimators of one
of the nuisance parameters. The mixed bias class is a strict superset of
both the class in \cite{chernozhukovclass} and \cite{jamieclass}. See \cite%
{characterization} for a discussion of the relation between these classes of
parameters, and examples covering many important parameters of interest in
causal inference and missing data.

A foundational precedent for this line of work was provided by \cite{robins2007comment}, who studied estimation of the mean outcome under missing-at-random
assumptions, equivalently, the mean of a counterfactual outcome under a point intervention assuming no unmeasured confounders. They showed that, with suitably chosen nuisance estimators, the
one-step estimator can collapse algebraically to either the outcome
regression estimator or the inverse probability weighted (IPW) estimator. In
the present note, we generalize these identities to the mixed bias framework
and also clarify how the new results in \cite{bruns2025augmented} relate to
those of \cite{robins2007comment}.

The remainder of the note is organized as follows. Section \ref{sec:mixed} reviews the
mixed bias class. Section \ref{sec:robins} summarizes and extends the identities in
\cite{robins2007comment} to the mixed bias framework. Section \ref{sec:bruns} extends the
identities in \cite{bruns2025augmented} to the mixed bias framework,
allowing for one of the nuisance estimators to be non-linear. The section
also clarifies the connection with the identities in \cite{robins2007comment}.

\section{The mixed bias class}\label{sec:mixed}

Suppose we observe a random sample of $n$ i.i.d copies of a random vector $O$
with unknown law $P_{0},$ and let $Z$ be a subvector of $O$. We assume that $%
P_{0}$ lies in non-parametric model $\mathcal{M}$, meaning that the tangent
space at each $P\in \mathcal{M}$ coincides with $L^{2}(P).$ Our goal is to
estimate a scalar parameter $\chi \left( P_{0}\right) $ based on the sample.
We assume that the functional $P\mapsto \chi \left( P\right) $ is pathwise
differentiable for all $P$ in $\mathcal{M}$. Because the model $\mathcal{M}$
is non-parametric, $\chi $ admits a unique influence function $\chi _{P}^{1}$
at each $P.$  See Chapter 25 of \cite{van2000} for the definitions of
pathwise differentiable parameters, tangent spaces and influence functions.

The parameter $\chi $ is said to have the mixed bias property over model $%
\mathcal{M}$ if there exists a statistic $S_{ab}=s_{ab}(O)$ such that for
all $P,P^{\prime }\in \mathcal{M}$, there exist functions $a_{P},b_{P}\in
L^{2}(P_{Z})$, $a_{P^{\prime }},b_{P^{\prime }}\in L^{2}(P_{Z})$ such that 
\begin{equation*}
\chi (P^{\prime })-\chi (P)+E_{P}\left\{ \chi _{P^{\prime }}^{1}(O)\right\}
=E_{P}\left[ S_{ab}\left\{ a_{P^{\prime }}(Z)-a_{P}(Z)\right\} \left\{
b_{P^{\prime }}(Z)-b_{P}(Z)\right\} \right] .
\end{equation*}%
Let $\mathcal{A}=\{a_{P}:P\in \mathcal{M}\}$ and $\mathcal{B}=\{b_{P}:P\in 
\mathcal{M}\}$. Theorem 1 of \cite{characterization}, established that,
under mild regularity conditions, $\chi $ has the mixed bias property if and
only if, there exist linear maps $h\in \mathcal{A}\rightarrow m_{1}(O,h)$
and $h\in \mathcal{B}\rightarrow m_{2}(O,h)$ and a statistic $S_{0}=s_{0}(O)$
such that for all $P\in \mathcal{M}$, 
\begin{equation*}
\chi
_{P}^{1}(O)=S_{ab}a_{P}(Z)b_{P}(Z)+m_{1}(O,a_{P})+m_{2}(O,b_{P})+S_{0}-\chi
(P).
\end{equation*}%
Moreover, Theorem 2 of \cite{characterization} established that if $\chi $
has the mixed bias property then, under mild regularity conditions, it holds
that

\begin{enumerate}
\item 
\begin{align*}
\chi(P)&=E_{P}\left\lbrace m_{1}(O,b_{P})\right\rbrace + E_{P}(S_0) \\
&=E_{P}\left\lbrace m_{2}(O,b_{P})\right\rbrace+ E_{P}(S_0) \\
&=- E_{P}\left\lbrace S_{ab}a_{P}(Z)b_{P}(Z)\right\rbrace + E_{P}(S_0)
\end{align*}

\item If $E(S_{ab}\mid Z)\neq 0$ almost surely under $P$ then 
\begin{equation*}
a_{P}(Z)=-r_{2,P}(Z)/E(S_{ab}\mid Z),
\end{equation*}
where $r_{2,P}(Z)$ is the Riesz representer of the map $h\to E_{P}\lbrace
m_{2}(O,h)\rbrace$. 

\item If $E(S_{ab}\mid Z)\neq 0$ almost surely under $P$ then 
\begin{equation*}
b_{P}(Z)=-r_{1,P}(Z)/E(S_{ab}\mid Z),
\end{equation*}
where $r_{1,P}(Z)$ is the Riesz representer of the map $h\to E_{P}\lbrace
m_{1}(O,h)\rbrace$.
\end{enumerate}

In what follows, we will assume that $E(S_{ab}\mid Z)\neq 0$ almost surely
under $P$, for all $P\in \mathcal{M}$. Moreover, to simplify the
presentation, and without loss of generality, we will assume that $S_0=0$.

The class of parameters studied in \cite{bruns2025augmented} coincides with
the subset of the mixed bias class of parameters that have $S_{ab}=-1$, $%
m_2(O,b)=Yb(Z)$ and $a_{P}(Z)=E_{P}(Y\mid Z)$, where $Y$ is a coordinate of $O$ that is not a part of $Z$.

For any functions $a$ and $b$ we define 
\begin{equation*}
\widehat{\chi }_{a,b}:=\mathbb{P}_{n}\left\{ S_{ab}a(Z)b(Z)\right\} +\mathbb{%
P}_{n}\left\{ m_{1}(O,a)\right\} +\mathbb{P}_{n}\left\{ m_{2}(O,b)\right\} 
\end{equation*}%
where $\mathbb{P}_{n}$ in the empirical mean operator. The estimator $%
\widehat{\chi }_{\widehat{a},\widehat{b}}$, where $\widehat{a}$ and $\widehat{b}$ are estimators of
the nuisance functions $a_{P}$ and $b_{P}$, is the so-called one-step estimator of $%
\chi (P)$. The mixed bias property of the
functional $\chi \left( P\right) $ implies that  $\widehat{\chi }_{\widehat{a%
},\widehat{b}}$ is doubly-robust consistent: it converges in probability to $%
\chi \left( P_{0}\right) $ if either $\widehat{a}$ consistently estimates $%
a_{P_{0}}$ or $\widehat{b}$ consistent estimates $b_{P_{0}},$ but not
necessarily both. For the class of parameters studied in \cite%
{bruns2025augmented}, $\widehat{\chi }_{\widehat{a},\widehat{b}}$ coincides
with the estimators they termed \textit{augmented estimators}.

By the linearity of $m_{2}$, 
\begin{equation*}
\widehat{\chi }_{a,0}=\mathbb{P}_{n}\{m_{1}(O,a)\}.
\end{equation*}
When $\chi $ belongs to the class of parameters considered in \cite%
{bruns2025augmented}, $\widehat{\chi }_{\widehat{a},0}$ coincides with what
the authors termed \textit{outcome regression} estimators. To maintain
connection with their work, we will refer to estimators of the form $%
\widehat{\chi }_{\widehat{a},0}$ as \textit{outcome regression-type
estimators}. Similarly, we will refer to $\widehat{\chi }_{0,\widehat{b}}$
as an \textit{IPW-type estimator}. This terminology is motivated by the fact
that, in the special case where $\chi (P)$ identifies the mean of a counterfactual outcome under a point intervention assuming no unmeasured confounders, $\widehat{\chi }%
_{0,\widehat{b}}$ reduces to the inverse probability weighted (IPW)
estimator of $\chi \left( P\right) $ with $\widehat{b}(Z)$ serving as an
estimator of $b_{P}(Z)=A/P(A\mid L)$, where $Z=(A,L),$ $A$ is a binary
treatment indicator and $L$ is the vector of confounders.

\section{The algebraic identities in \cite{robins2007comment}}\label{sec:robins}

\cite{robins2007comment} showed that, in the estimation of the counterfactual outcome
mean, the one-step estimator $\widehat{\chi }_{\widehat{a},\widehat{b}}$
reduces to the outcome regression estimator $\widehat{\chi }_{\widehat{a},0}$
when $\widehat{a}$ satisfies a certain estimating equation. They referred to
such estimators as regression double robust estimators. Similarly, $\widehat{%
\chi }_{\widehat{a},\widehat{b}}$ reduces to an IPW estimator $\widehat{\chi 
}_{0,\widehat{b}}$ when $\widehat{b}$ satisfies a certain estimating
equation, yielding to what they termed Horvitz-Thompson double robust
estimators. In their setting the observed data vector in each sample unit is 
$O=\left( A,L,Y\right) ,$ with $A$ a binary treatment, $L$ a vector of
always-observed covariates and $Y$ the factual outcome. Letting $Z=\left(
A,L\right) \,\ $and under the standard assumptions of consistency,
positivity and no-unmeasured confounding, the mean of the outcome when
treatment $A$ is set to 1 in the entire population is equal to $\chi \left(
P\right) $ where $m_{1}\left( O,a\right) =a\left( A=1,L\right) ,m_{2}\left(
O,b\right) =b\left( Z\right) Y$ with nuisance functions $a_{P}\left(
Z\right) =E_{P}\left( Y|A,L\right) $ and $b_{P}\left( Z\right) =AY/P\left(
A=1|L\right) .$ 

\cite{robins2007comment} key identities extend to the general mixed bias class of
parameters as follows:

\begin{enumerate}
\item If $\widehat{a}$ satisfies 
\begin{equation}
\mathbb{P}_{n}\left\{ S_{ab}\widehat{a}(Z)\widehat{b}(Z)+m_{2}(O,\widehat{b}%
)\right\} =0  \label{l1}
\end{equation}%
then $\widehat{\chi }_{\widehat{a},\widehat{b}}=\mathbb{P}_{n}\left\{
m_{1}(O,\widehat{a})\right\} =\widehat{\chi }_{\widehat{a},0}.$ So, by
choosing the estimator $\widehat{a}$ of the nuisance $a_{P}$ such that it
satisfies the equation in the last display we ensure that the outcome
regression-type estimator $\widehat{\chi }_{\widehat{a},0}$ is algebraically
identical to the one-step estimator $\widehat{\chi }_{\widehat{a},\widehat{b}%
}$

\item If $\widehat{b}$ satisfies 
\begin{equation}
\mathbb{P}_{n}\left\{ S_{ab}\widehat{a}(Z)\widehat{b}(Z)+m_{1}(O,\widehat{a}%
)\right\} =0  \label{l2}
\end{equation}%
then $\widehat{\chi }_{\widehat{a},\widehat{b}}=\mathbb{P}_{n}\left\{
m_{2}(O,\widehat{b})\right\} =\widehat{\chi }_{0,\widehat{b}}.$ So, by
choosing the estimator $\widehat{b}$ of the nuisance $b_{P}$ such that it
satisfies the equation in the last display we ensure that the IPW-type
estimator $\widehat{\chi }_{0,\widehat{b}}$ is algebraically identical to
the one-step estimator $\widehat{\chi }_{\widehat{a},\widehat{b}}.$
\end{enumerate}

The proof of these results is simply a consequence of the fact that $%
\widehat{\chi }_{\widehat{a},\widehat{b}}$ is a sum of three terms; if the
sum of two of these three terms is exactly zero then $\widehat{\chi }_{%
\widehat{a},\widehat{b}}$ must equal the remaining term. \cite{robins2007comment}
discussed several ways of constructing, not necessarily linear, estimators $%
\widehat{a}$ satisfying \eqref{l1}, and estimators $\widehat{b}$ satisfying %
\eqref{l2}.  

Of course, if $\widehat{a}$ and $\widehat{b}$ satisfy simultaneously both
equations then the one-step estimator $\widehat{\chi }_{\widehat{a},\widehat{%
b}}$ is simultaneously an outcome-regression type estimator $\widehat{\chi }%
_{\widehat{a},0}$ and IPW-type estimator $\widehat{\chi }_{0,\widehat{b}}$.
It is possible to construct $\widehat{a}$ and $\widehat{b},$ linear in $\phi
\left( Z\right) ,$ where 
\begin{equation*}
\phi \left( Z\right) :=[\phi _{1}\left( Z\right) ,...,\phi _{d}\left(
Z\right) ]^{\top }
\end{equation*}%
is a vector of given features, that satisfy \eqref{l1} and \eqref{l2}
simultaneoulsy. To see this, note that for a linear estimator $\widehat{a}%
(Z)=\widehat{\beta }^{T}\phi \left( Z\right) $ condition \eqref{l1} holds
whenever $\widehat{\beta }$ solves the equation 
\begin{equation}
\mathbb{P}_{n}\left\{ S_{ab}\widehat{b}(Z)\phi \left( Z\right) ^{T}\beta
+m_{2}(O,\widehat{b})\right\} =0.  \label{main-alpha-eq}
\end{equation}%
Then, by linearity of $m_{2},$ when $\widehat{b}(Z)=\widehat{\alpha }%
^{T}\phi \left( Z\right) $ is also linear in $\phi \left( Z\right) ,$
equation \eqref{main-alpha-eq} is the same as 
\begin{equation*}
\widehat{\alpha }^{T}\mathbb{P}_{n}\left\{ S_{ab}\phi \left( Z\right) \phi
\left( Z\right) ^{T}\beta +m_{2}(O,\phi )\right\} =0
\end{equation*}

Thus, in particular, 
\begin{equation*}
\widehat{\beta }_{OLS}:=\mathbb{P}_{n}\left\{ -S_{ab}\phi \left( Z\right)
\phi \left( Z\right) ^{T}\right\} ^{-1}\mathbb{P}_{n}\left\{ m_{2}(O,\phi
)\right\} 
\end{equation*}%
satisfies equation \eqref{main-alpha-eq}. We call $\widehat{\beta }_{OLS}$
the OLS type estimator of $\beta $ because in the special case in which $%
m_{2}(O,b)=Yb$ and $S_{ab}=-1$, $\widehat{\beta }_{OLS}$ is the ordinary
least squares estimator of $\beta $ in the regression of $Y$ on predictors $%
\phi (Z).$ It follows that $\widehat{\chi }_{\widehat{a}_{OLS},\widehat{b}}$ 
$=\widehat{\chi }_{\widehat{a}_{OLS},0}$ $\ $where $\widehat{a}_{OLS}(Z):=%
\widehat{\beta }_{OLS}^{T}\phi \left( Z\right) .$ If, additionally, $%
\widehat{\alpha }=\widehat{\alpha }_{OLS}:=\mathbb{P}_{n}\left\{ -S_{ab}\phi
\left( Z\right) \phi \left( Z\right) ^{T}\right\} ^{-1}\mathbb{P}_{n}\left\{
m_{1}(O,\phi )\right\} ,$ then we have that $\widehat{\chi }_{\widehat{a}%
_{OLS},\widehat{b}_{OLS}}=\widehat{\chi }_{\widehat{a}_{OLS},0}=\widehat{%
\chi }_{0,\widehat{b}_{OLS}}$ where $\widehat{b}_{OLS}(Z):=\widehat{\alpha }%
_{OLS}^{T}\phi \left( Z\right) .$ 

\section{The algebraic identities in \cite{bruns2025augmented}}\label{sec:bruns}

% The following result generalizes the result in Proposition 3.1 of
% \cite{bruns2025augmented}. 

Let $\widehat{b}$ be \textbf{any} estimator of $b_{P}$ (not necessarily
linear) and let $\widetilde{a}\left( Z\right) =\widetilde{\beta }^{T}\phi
\left( Z\right) $ be an estimator of $a_{P}$, where $\widetilde{\beta }$ is
any vector that solves \eqref{main-alpha-eq}.\ Then, the following hold: 
\begin{align}
\mathbb{P}_{n}\{m_{2}(O,\widehat{b})\}& =\mathbb{P}_{n}\{-S_{ab}\widehat{b}%
(Z)\widetilde{a}(Z)\}  \label{eq:basic} \\
& =\widehat{\Phi }_{\widehat{b}}^{\top }\widetilde{\beta },
\label{eq:linear_transform}
\end{align}%
where $\widehat{\Phi }_{\widehat{b}}=\mathbb{P}_{n}\{-S_{ab}\widehat{b}%
(Z)\phi (Z)\}$. The identity in \eqref{eq:linear_transform} shows that any
estimator of $\chi $ of the form $\mathbb{P}_{n}\left\{ m_{2}(O,\widehat{b}%
)\right\} $, for an arbitrary $\widehat{b}$, can be expressed as a linear
function of the vector $\widehat{\Phi }_{\widehat{b}}$.

Proposition 3.1 of \cite{bruns2025augmented} is precisely the identity %
\eqref{eq:linear_transform} for the special case in which $m_{2}(O,b)=Yb,$ $%
S_{ab}=-1,$ $\widehat{b}(Z)=\widehat{\alpha }^{T}\phi \left( Z\right) $ with 
$\widehat{\alpha }$ arbitrary and $\widetilde{a}=\widehat{a}_{OLS}$ as
defined in the preceding section. 

We are now ready to establish the main result of this note, which extends
Proposition 3.2 of \cite{bruns2025augmented} to the mixed bias class of
parameters, allowing for arbitrary estimators of $\widehat{b}$.

\begin{proposition}\label{prop:main}
Let $\widehat{\beta }^{\top }\in \mathbb{R}^{p}$ be \textbf{any vector} and 
\begin{equation*}
\widehat{a}(Z):=\widehat{\beta }^{\top }\phi (Z).
\end{equation*}%
and let $\widehat{b}$ be \textbf{any} estimator of $b_{P}$. Let 
\begin{equation*}
\widetilde{a}_{aug}\left( Z\right) :=\widetilde{\beta }_{aug}^{\top }\phi (Z)
\end{equation*}%
where%
\begin{equation*}
\widetilde{\beta }_{aug,j}:=(1-\gamma _{j})\widehat{\beta }_{j}+\gamma _{j}%
\widetilde{\beta }_{j},
\end{equation*}%
with $\widetilde{\beta }_{j}$ the $j^{th}$ entry of a vector $\widetilde{%
\beta }$ that satisfies equation \eqref{main-alpha-eq} and 
\begin{equation*}
\gamma _{j}:=\frac{\mathbb{P}_{n}\left\{ -S_{ab}\widehat{b}(Z)\phi
_{j}(Z)\right\} }{\mathbb{P}_{n}\left\{ m_{1}(O,\phi _{j})\right\} }.
\end{equation*}%
Then, 
\begin{equation}
\widehat{\chi }_{\widehat{a},\widehat{b}}=\widehat{\chi }_{\widetilde{a}%
_{aug},0}  \label{eq:main_aug}
\end{equation}
\end{proposition}

\begin{proof}
\begin{align*}
\widehat{\chi }_{\widetilde{a}_{aug},0}& =\mathbb{P}_{n}\left\{ m_{1}(O,%
\widetilde{a}_{aug})\right\}  \\
& =\mathbb{P}_{n}\left\{ m_{1}(O,\phi )^{\top }\right\} \widetilde{\beta }%
_{aug} \\
& =\sum\limits_{j=1}^{p}\mathbb{P}_{n}\left\{ m_{1}(O,\phi _{j})(1-\gamma
_{j})\widehat{\beta }_{j}+m_{1}(O,\phi _{j})\gamma _{j}\widetilde{\beta }%
_{j}\right\}  \\
& =\mathbb{P}_{n}\left[ \sum\limits_{j=1}^{p}m_{1}(O,\phi _{j})\widehat{%
\beta }_{j}\right] +\sum\limits_{j=1}^{p}\mathbb{P}_{n}\left\{ m_{1}(O,\phi
_{j})\frac{\mathbb{P}_{n}\left\{ S_{ab}\widehat{b}(Z)\phi _{j}(Z)\right\} }{%
\mathbb{P}_{n}\left\{ m_{1}(O,\phi _{j})\right\} }\left( \widehat{\beta }%
_{j}-\widetilde{\beta }_{j}\right) \right\}  \\
& =\mathbb{P}_{n}\left\{ m_{1}(O,\widehat{a})\right\} +\sum\limits_{j=1}^{p}%
\mathbb{P}_{n}\left\{ S_{ab}\widehat{b}(Z)\phi _{j}(Z)\right\} \left( 
\widehat{\beta }_{j}-\widetilde{\beta }_{j}\right)  \\
& =\mathbb{P}_{n}\left\{ m_{1}(O,\widehat{a})\right\} +\mathbb{P}_{n}\left\{
S_{ab}\widehat{a}(Z)\widehat{b}(Z)\right\} -\mathbb{P}_{n}\left\{ S_{ab}%
\widetilde{a}(Z)\widehat{b}(Z)\right\}  \\
& =\mathbb{P}_{n}\left\{ m_{1}(O,\widehat{a})\right\} +\mathbb{P}_{n}\left\{
S_{ab}\widehat{a}(Z)\widehat{b}(Z)\right\} +\mathbb{P}_{n}\left\{ m_{2}(O,%
\widehat{b})\right\}  \\
& =\widehat{\chi }_{\widehat{a},\widehat{b}}
\end{align*}%
where the next to last equality holds by \eqref{eq:basic}.
\end{proof}

In the special case where $m_{2}(O,b)=Yb$, $S_{ab}=-1$, and $\widehat{b}$ is
linear in $\phi (Z)$, equation \eqref{eq:main_aug} reduces precisely to the
conclusion of Proposition 3.2 of \cite{bruns2025augmented}.

Proposition \ref{prop:main} is noteworthy because it characterizes the one-step
estimator $\widehat{\chi }_{\widehat{a},\widehat{b}}$ constructed with
arbitrary linear estimators $\widehat{a}$ and $\widehat{b}$ as a particular
outcome regression-type estimator $\widehat{\chi }_{\widetilde{a}_{aug},0}$.
By symmetry, one can analogously deduce that $\widehat{\chi }_{\widehat{a},%
\widehat{b}}$ is equal to an IPW-type estimator $\widehat{\chi }_{0,%
\widetilde{b}_{aug}}$ for some $\widetilde{b}_{aug}$ defined analogously to $%
\widetilde{a}_{aug},$ by interchanging the roles of $\widehat{a}$ and $%
\widehat{b}$ and swapping $m_{1}$ with $m_{2}.$ In particular, one obtains
direct characterizations of one-step estimators that use linear nuisance
estimators  with coefficients estimated by Ridge regression, or by $\ell _{1}
$ or $\ell _{\infty }$ regularization, as specific instances of this general
form. \cite{bruns2025augmented} derived explicit expressions for $\widetilde{%
a}_{aug}$ in these cases for the parameter class they studied; these
formulas extend immediately to the general class of mixed bias parameters,
and we therefore omit the derivations here. While \cite{bruns2025augmented}
argued that this characterization should facilitate the study of asymptotic
properties of one-step estimators, it remains unclear to us how this
argument can be fully carried out.  

Finally, we conclude this note highlighting the connection between
Proposition \ref{prop:main} and the results presented in Section \ref{sec:robins}. The coefficients 
$\gamma _{j}$ can be interpreted as quantifying the imbalance of the
estimator $\widehat{b}$ relative to the covariate functions $\phi _{j}.$
Perfectly balanced estimators $\widehat{b}$ correspond to $\gamma _{j}=1$
for all $j=1,...,d$, in which case $\widetilde{a}_{aug}=\widetilde{a}.$
Moreover, perfectly balanced estimators $\widehat{b}$ satisfy the equation %
\eqref{l2} for any $\widehat{a}$ linear in $\phi \left( Z\right) $. Hence,
in this setting, $\widehat{\chi }_{\widehat{a},\widehat{b}}=\widehat{\chi }%
_{0,\widehat{b}}$ is the same for all $\widehat{a}$ linear in $\phi \left(
Z\right) .$ By Proposition \ref{prop:main} it then follows that $\widehat{\chi }_{%
\widehat{a},\widehat{b}}=\widehat{\chi }_{0,\widehat{b}}=\widehat{\chi }_{%
\widetilde{a},0},$ reproducing the result in Section \ref{sec:robins} that an outcome
regression-type coincides with an IPW-type estimator whenever equations %
\eqref{l1} and \eqref{l2} both hold, since by definition $\widetilde{a}$
satisfies equation \eqref{l1} with $\widetilde{a}$ in place of $\widehat{a}.$

\section{Acknowledgements}
Andrea Rotnitzky was funded by NIH grants UM1 AI068635, R01HL137808 and R37
Al029168.

\bibliography{augmented}

\begin{thebibliography}{}

\bibitem[Bruns-Smith et~al., 2025]{bruns2025augmented}
Bruns-Smith, D., Dukes, O., Feller, A., and Ogburn, E.~L. (2025).
\newblock Augmented balancing weights as linear regression.
\newblock {\em Journal of the Royal Statistical Society Series B: Statistical Methodology}, page qkaf019.

\bibitem[Chernozhukov et~al., 2022]{chernozhukovclass}
Chernozhukov, V., Newey, W.~K., and Singh, R. (2022).
\newblock Automatic debiased machine learning of causal and structural effects.
\newblock {\em Econometrica}, 90(3):967--1027.

\bibitem[Robins et~al., 2008]{jamieclass}
Robins, J., Li, L., Tchetgen, E., van~der Vaart, A., et~al. (2008).
\newblock Higher order influence functions and minimax estimation of nonlinear functionals.
\newblock In {\em Probability and statistics: essays in honor of David A. Freedman}, volume~2, pages 335--422. Institute of Mathematical Statistics.

\bibitem[Robins et~al., 2007]{robins2007comment}
Robins, J., Sued, M., Lei-Gomez, Q., and Rotnitzky, A. (2007).
\newblock Comment: Performance of double-robust estimators when" inverse probability" weights are highly variable.
\newblock {\em Statistical Science}, 22(4):544--559.

\bibitem[Rotnitzky et~al., 2021]{characterization}
Rotnitzky, A., Smucler, E., and Robins, J.~M. (2021).
\newblock Characterization of parameters with a mixed bias property.
\newblock {\em Biometrika}, 108(1):231--238.

\bibitem[Van~der Vaart, 2000]{van2000}
Van~der Vaart, A.~W. (2000).
\newblock {\em Asymptotic statistics}, volume~3.
\newblock Cambridge university press.

\end{thebibliography}
\bibliographystyle{apalike}
\end{document}